\pdfoutput=1
\documentclass[10pt,a4paper]{article}
\usepackage{a4wide}
\usepackage{amsmath, amssymb, amsthm}
\usepackage[utf8]{inputenc}
\usepackage[english]{babel}
\usepackage[numbers]{natbib}
\usepackage{url}
\usepackage[usenames]{color}
\usepackage[colorlinks=true,
linkcolor=webgreen,
filecolor=webbrown,
citecolor=webgreen]{hyperref}
\definecolor{webgreen}{rgb}{0,.5,0}
\definecolor{webbrown}{rgb}{.6,0,0}
\definecolor{red}{rgb}{1,0,0}
\usepackage{cleveref}
\usepackage{tabularx}
\usepackage{breqn}
\usepackage[title]{appendix}
\theoremstyle{plain}

\newtheorem{lemma}{Lemma}[section]

\newtheorem{theorem}{Theorem}[section]
\newtheorem*{theorem*}{Theorem}

\crefname{conjecture}{Conjecture}{Conjectures}
\crefname{theorem}{Theorem}{Theorems}
\crefname{theorem*}{Theorem}{Theorems}
\crefname{corollary}{Corollary}{Corollaries}
\crefname{lemma}{Lemma}{Lemmas}
\crefname{proposition}{Proposition}{Propositions}
\crefname{remark}{Remark}{Remarks}
\crefname{note}{Note}{Notes}
\crefname{todo}{TODO}{TODOs}
\crefname{definition}{Definition}{Definitions}
\crefname{notation}{Notation}{Notations}
\crefname{example}{Example}{Examples}
\crefname{question}{Question}{Questions}
\crefname{section}{\S}{Sections}
\crefname{equation}{Equation}{Equations}
\crefformat{equation}{(eq.~#2#1#3)}

\newcommand{\floor}[1]{\left\lfloor #1 \right\rfloor}

\newcommand{\HW}{\operatorname{HW}}

\newcommand{\Z}{\mathbb{Z}}

\DeclareMathOperator{\DTIME}{DTIME}
\DeclareMathOperator{\ELEMENTARY}{ELEMENTARY}

\title{How many points does an affine algebraic set have in residue classes modulo $n$ ?}
\author{Mihai Prunescu \footnote{Research Center for Logic, Optimization and Security (LOS), Faculty of Mathematics and Computer Science, University of Bucharest, Academiei 14, Bucharest (RO-010014), Romania.} \footnote{ Simion Stoilow Institute of Mathematics of the Romanian Academy, Research unit 5, P. O. Box 1-764, Bucharest (RO-014700), Romania. E-mail: {\tt mihai.prunescu@imar.ro}} 
}

\begin{document}
\date{}
\maketitle

\begin{abstract} \noindent
We show that for every uniform family of affine algebraic sets, there is a formula in the parameters of the family and in $n$, expressing the cardinality of the set in the ring  $\mathbb Z / n \mathbb Z$. In particular, we construct such a formula for the elliptic curves in Weierstra{\ss }  normal form.  These formulas are arithmetic terms: fixed finite compositions of additions, multiplications, subtractions, divisions with remainder and integer exponentiations.
\\[2mm]
\end{abstract}

\section{Introduction}
The \textbf{Kalmar functions}, also known as \textbf{elementary functions} and denoted below as $\ELEMENTARY$, are the sequences of non-negative integers which can be deterministically computed in iterated exponential time (Marchenkov \cite[Introduction]{marchenkov2007superposition}):
\begin{align*}
\ELEMENTARY =  \bigcup_{\textrm{all heights}}\DTIME\left (2^{2^{\cdot^{\cdot^{\cdot^{\cdot ^{2^{2^m}}}}}}} \right ) ,
\end{align*}
where $m$ is the length of the input. Here $\DTIME(g(m))$ means the class of functions $f:\mathbb N^k \rightarrow \mathbb N$ which are computed by determinist Turing machines in a number of steps bounded by $g(m)$, where $m$ is the length of the input. Any Kalmar elementary function is computable in deterministic time bounded by a fixed number of iterated exponential functions. Mazzanti \cite{mazzanti2002plainbases} showed that Kalmar functions $f : \mathbb N^k \rightarrow \mathbb N$ are expressible as arithmetic terms. An arithmetic term can be obtained by composition and substitution using the elementary arithmetic operations:
$$x + y, x \dot{-} y, x \cdot y, \left \lfloor \frac{x}{y} \right \rfloor, 2^x,$$
where $x\dot{-}y$ is the arithmetic difference defined by $x \dot{-} y = 0$ when $x < y$ and $x \dot{-}y = x - y$ if $x \geq y$.  A proof, which is similar with Mazzanti's, was given by Marchenkov, \cite{marchenkov2007superposition}. These results were applied by Prunescu and Sauras-Altuzarra to construct terms for some number-theoretic functions, \cite{prunescu2024numbertheoreticfunctions}. It is known that the exponentials $k^x$ and the two-variable exponentiation $x^y$ are also expressible as arithmetic terms in these functions, so these exponentiations can also be used in the term construction.

Many interesting functions have been represented by arithmetic terms. Two notable examples are the prime counting function $\pi(n)$ and the $n$-th prime number $p(n)$, Prunescu and Shunia \cite{prunescushuniaprimes}. For many other applications, Istrate, Prunescu and Shunia \cite{istrateprunescushuniauniversal} and further papers cited therein. 

We apply this technique to show that for every family of algebraic sets defined with integer coefficients and parameters $\vec \lambda \in \mathbb Z$ there is an arithmetic term $f(\vec \lambda, n)$ giving the number of points of the set defined in $(\mathbb Z / n \mathbb Z)^k$. We treat also the particular case of the elliptic curve. This is the subset of $(\mathbb Z / n \mathbb Z)^2$
$$x_2^2 = x_1^3 + Ax_1 + B$$
with parameters $A, B \in \mathbb Z$. Its cardinality will be expressed by an arithmetic term $f(A, B, n)$. 

\section{Preliminaries and definitions}\label{section2}
In this section, we display various arithmetic terms that are used in our results. The main result of the section is that the number of solutions of an exponential Diophantine equation $P(x_1, \dots, x_k) = 0$ in the cube $\{0, 1, \dots, t-1\}^k$ is represented by an arithmetic term in the coefficients of $P(\vec x)$ and in $t$. This {\bf hypercube method} has been first applied by Mazzanti \cite{mazzanti2002plainbases}, then by Marchenkov \cite{marchenkov2007superposition}. It has been also applied in different papers by the author, for example in \cite{prunescu2024numbertheoreticfunctions}. In this section we show how the method works, without some details.

\subsection{Number theoretic arithmetic terms}
The following number theoretic arithmetic terms have been gathered from various sources \cite{robinson1952arithmetic, mazzanti2002plainbases, marchenkov2007superposition, prunescushuniagcd}:
\begin{align*}
\gcd(a,b) &= \left (\floor{\frac{5^{ab(ab + a+b)}}{(5^{a^2b}-1)(5^{b^2a}-1)}} \bmod 5^{ab} \right ) \dot{-} 1, \\
\nu_2(n) &= \floor{\frac{\gcd(n, 2^n)^{n+1} \bmod (2^{n+1}-1)^2}{2^{n+1}-1}} , \\
\binom{n}{k} &= \floor{\frac{(2^n+1)^n}{2^{nk}}} \bmod 2^n , \\
\HW(n) &= \nu_2 \left( \binom{2n}{n} \right).
\end{align*}
Here, $\nu_2(n)$ represents the $2$-adic order of $n$, which is highest exponent of $2$ dividing $n$. $\HW(n)$ denotes the Hamming weight of $n$, which is the number of ones in the binary representation of $n$. We consider $a, b, n, k > 0$ to avoid division by zero. On the other hand, arithmetic terms are considered to be totally defined functions, and such cases can be solved by adopting various conventions.

\subsection{Generalized geometric progressions} \label{subsection:geoemtricprogressions}
Consider $a,r,b \in \Z^+$ such that $a > 1$, $r \geq 0$ and $b \geq 0$. The arithmetic term for the geometric progression
\begin{align*}
G_0(a,b) = \sum_{j=0}^b a^j = \frac{a^{b+1}-1}{a-1}
\end{align*}
is well-known. Perhaps lesser known, are the \textbf{generalized geometric progressions of the $r$-th kind}, which are defined as
\begin{align*}
G_r (a, b) = \sum_{j=0}^b j^r a^j .
\end{align*}
As described by Matiyasevich in the appendix of \cite{matiyasevich1993hilbert}, $G_r (q, t)$ can be calculated via the formula
\begin{align} \label{TermG}
 G_{r + 1} ( a , b ) = \frac{\partial}{\partial a} G_r ( a , b + 1 ) - \sum_{j=0}^r \binom{r+1}{j} G_j ( a , b ) .
\end{align} 
For example, 
 $$ G_1(a,b) =\dfrac{b a^{t+2} - (b+1)a^{b + 1} + a}{( a - 1 )^2} , $$ $$ G_2(a,b) = \dfrac{b^2 a^{b+3} - (2 b^2 + 2b - 1)a^{b+2} + (b+1)^2 a^{b+1} - a^2 - a}{( a - 1 )^3} .$$ 
In our application to elliptic curves, terms up to $G_{12}(a, b)$ occur, excepting $G_{11}(a,b)$.

\subsection{The number of solutions} 
\label{subsection:numberofsolutions} 

Given two integers $ a $ and $ w $ such that $ 0 \leq a < 2^w $, define $ \delta ( a , w ) $ as $$ \delta(a,w) := ( 2^w - 1 ) ( 2^w - a + 1 ) = 2^{2 w} - 2^w a + a - 1 .$$
We observe that:
$$\HW ( \delta ( a , w ) ) = \begin{cases} 2w, & a = 0, \\ w, & a \neq 0.\end{cases}$$
Consider the points with integer coordinates, contained in the $k$-dimensional cube $[0, t-1]^k$ and a function $f : [0, t-1]^k \cap \mathbb N^k \rightarrow \mathbb N$. Suppose that for all $\vec x \in [0, t-1]^k \cap \mathbb N^k $, one has that $0 \leq f(\vec x) < 2^w$. We stress the fact that $f$ has to be non-negative. 

Let $ v $ denote the function that maps each point $ \vec{x} \in \{ 0 , \dots , t - 1 \}^k $ into  $ x_1 + x_2 t  + \dots + x_k t ^{k - 1} $. Observe that $ v $ enumerates the points of $ \{ 0 , \dots , t  - 1 \}^k $ from $ 0 $ to $ t^k - 1 $.

Let $$ M  = \sum_{\vec{a} \in \{ 0 , \dots , t  - 1 \}^k} 2^{2 w  v ( \vec{a} )} \delta ( f (\vec{a} ) , w ) . $$
We observe that the binary representation of $M$ is the concatenation of the binary representations of the numbers $\delta(f(\vec a), w)$.

Let $ d $ denote the cardinality of the set $ \{ \vec{a} \in \{ 0 , \dots , t - 1 \}^k : f ( \vec{a} ) = 0 \} $. It follows that:
$$\HW(M) = 2wd + (t^k - d)w, $$
$$d =\HW(M) / w - t^k. $$
So if $M$ could be expressed by arithmetic term in $t$ and $w$, then the number of zeros of the function $f$ would be expressed by such a term. We call {\bf monomial simple in $x$} any expression of the shape:
$$m = c v_1^{x_1} \cdots v_k^{x_k} x_1^{u_1} \cdots x_k^{u_k}$$
where $u_1, \dots, u_k \geq 0$ and $v_1, \dots, v_k \geq 1$ are integers and $c$ is an integer. An exponential polynomial simple in $\vec x$ is a sum of simple monomials. We apply the identity: 
$$ \sum_{\vec{x} \in \{ 0 , \dots , t - 1 \}^k} x_1^{u_1} v_1^{x_1} \dots x_k^{u_k} v_k^{x_k} = \left ( \sum _{x_1 = 0}^{t-1} x_1^{u_1} v_1^{x_1} \right ) \cdots \left ( \sum _{x_k = 0}^{t-1} x_k^{u_k} v_k^{x_k} \right ) =$$ $$ = G_{u_1} ( v_1 , t - 1 ) \dots G_{u_k} ( v_k , t - 1 ) . $$ 

Let $P(\vec x) = \varepsilon  + m_1 (  \vec{x} ) + \dots + m_r (  \vec{x} )$, with $\varepsilon$ constant term. We suppose that $P(\vec x)$ is always non-negative. To count the number of zeros of $P(\vec x)$, we start expressing $M$ as an arithmetic term:
 $$ M = \sum_{\vec{x} \in \{ 0 , \dots , t  - 1 \}^k} 2^{2 w  v ( \vec{x} )} \delta ( P (  \vec{x} ) , w  ) = $$ $$ \sum_{\vec{x} \in \{ 0 , \dots , t  - 1 \}^k} 2^{2 w  v ( \vec{x} )} \delta ( \varepsilon  + m_1 (  \vec{x} ) + \dots + m_r (  \vec{x} ) , w  ) = $$ $$ \sum_{\vec{x} \in \{ 0 , \dots , t  - 1 \}^k} 2^{2 w  v ( \vec{x} )} ( 2^{w } - 1 ) ( 2^{w } - \varepsilon  - m_1 (  \vec{x} ) - \dots - m_r (  \vec{x} ) + 1 ) = $$ $$ \sum_{\vec{x} \in \{ 0 , \dots , t  - 1 \}^k} 2^{2 w  v ( \vec{x} )} ( 2^{w } - 1 ) ( 2^{w } - \varepsilon + 1 ) + $$ $$ + \sum_{j = 1}^r \sum_{\vec{x} \in \{ 0 , \dots , t  - 1 \}^k} 2^{2 w  v ( \vec{x} )} ( 2^{w } - 1 ) ( - m_j (  \vec{x} ) ) . $$

The contribution of an exponential monomial $m =   c v_1^{x_1} \cdots v_k^{x_k} x_1^{u_1} \cdots x_k^{u_k}$ to $M$ is an expression of the shape:
$$A(m, k, t, w) = \sum_{\vec{x} \in \{ 0 , \dots , t  - 1 \}^k} 2^{2 w  v ( \vec{x} )} ( 2^{w } - 1 ) ( - m (  \vec{x} ) )  $$ $$ = - ( 2^{w } - 1 ) \,\, c \,\, G_{u_1} ( 2^{ 2 w } v_1 , t  - 1 ) G_{u_2} ( 2^{ 2 w t } v_2 , t  - 1 )\dots G_{u_k} ( 2^{ 2 w  {t }^{k - 1}} v_k , t  -1 )$$
which is an arithmetic term in $t$ and $w$. If the exponential polynomial contains the free term $\varepsilon \in \mathbb Z$, 
meaning that $v_1 = \dots = v_k = 1$ and $u_1 = \dots = u_k = 0$, its contribution simplifies to: 
$$C(\varepsilon , k, t, w) =  \sum_{\vec{x} \in \{ 0 , \dots , t  - 1 \}^k} 2^{2 w  v ( \vec{x} )} ( 2^{w } - 1 ) ( 2^{w } - \varepsilon + 1 ) = $$ $$ = ( 2^{w } - \varepsilon  + 1 ) ( 2^{2 w  {t }^k } - 1 ) / ( 2^{w } + 1 ).$$
Observe that even for $\varepsilon = 0$, the contribution of the free term is not zero. It follows that for  exponential polynomials simple in $\vec x$, the quantity $M$ defined above can be expressed by an arithmetic term in $t$ and $w$. 

We have shown that {\it if a non-negative exponential polynomial function simple in $\vec x$ is defined on the points with integer coordinates inside a $k$-dimensional cube $\{0, 1, \dots, t-1\}^k$ and is strictly bounded by $2^w$, then $d$, the number of zeros inside the cube, will be expressed by an arithmetic term in $t$ and $w$.} 

In practice, the various coefficients and exponential bases $c$ and $v_i$ from various exponential monomials will be arithmetic terms depending on some parameter tuple $\vec n = (n_1, \dots, n_s)$, {\bf but} the exponents $u_1, \dots, u_k$ will be {\bf constants} in every monomial. In this case, one computes appropriate bounds $t(\vec n)$ and $w(\vec n)$ such that all interesting zeros are in the cube $[0, t(n) -1]^k$ and the positive exponential polynomial function is bounded by $2^{w(\vec n)}$, and the number of natural number tuples, which are solutions of the equation $f(\vec x) = 0$, will be given by an arithmetic term $d(\vec n)$.

\section{Algebraic sets in residue classes modulo n}\label{sets}

Consider a system $S$ of polynomial equations with integer coefficients:
$$
    F_1(\vec \lambda, \vec x) = 0,
     \ldots\ldots, 
    F_k(\vec \lambda, \vec x) = 0,
$$
where $\vec \lambda = (\lambda_1, \dots, \lambda_s)$ are parameters and $\vec x = (x_1, \dots, x_m)$ are unknowns. For given values of the parameters $\vec \lambda \in \mathbb Z^s$ and for $n \geq 1$, consider the set:
$$S_n(\vec \lambda) = \{(x_1, \dots, x_m) \in (\mathbb Z / n \mathbb Z)^m \,|\, F_1(\vec \lambda, \vec x) = \dots = F_k(\vec \lambda, \vec x) = 0 \textrm{ in $\mathbb Z/ n\mathbb Z$} \}. $$ 
We will show that there is an arithmetic term $f(\vec \lambda , n)$ such that for all $\vec \lambda \in \mathbb Z^s$ and all $n \geq 1$,
$$f(\vec \lambda, n) = |S_n(\vec \lambda)|. $$ 
Consider the Diophantine equation $F(\vec \lambda, n, \vec x, \vec y, \vec z) = 0$ defined as:
$$\sum _{i = 1}^k {\left ( {F_i(\vec \lambda, \vec x)}^2 - n^2 y_i^2 \right )}^2 + \sum _{j = 1}^m {(x_j + z_j - n + 1)}^2 = 0.$$
For given  $\vec \lambda \in \mathbb Z ^s$ and $n \geq 1$, define the set:
$$T_n(\vec \lambda) = \{(x, y, z) \in \mathbb{N}^{m} \times \mathbb{N}^{k} \times \mathbb{N}^{m} \,|\, F(\vec \lambda, n, \vec x, \vec y, \vec z) = 0 \}. $$ 
\begin{lemma}\label{first}
    Given the system $S$, for all $\vec \lambda \in \mathbb Z ^s$ and $n \geq 1$, 
    $$|S_n(\vec \lambda)| = |T_n(\vec \lambda)| .$$
\end{lemma}

\begin{proof} Fix the values of $n \geq 1$ and $\vec \lambda \in \mathbb Z^s$. Consider the application $r : S_n(\vec \lambda) \rightarrow T_n(\vec \lambda)$ given by:
$$r(\vec x) = (\vec x, \vec y, \vec z),$$
where every class $x_j \bmod n$ corresponds to the unique representative $x_j$ with $0 \leq x_j \leq n-1$, every $y_i$ has the value:
$$y_i = \frac{|F_i (\vec \lambda, \vec x)|}{n},$$
and every $z_j$ has the value:
$$z_j = n - 1 - x_j.$$
If $\vec x \in S_n(\vec \lambda)$ then $r(x) \in T_n(\vec \lambda)$. This function is one-to-one and onto because the values of $\vec y$ and $\vec z$ are uniquely determined by $\vec x$. 
\end{proof}

In the following lines we will use the arithmetic term:
$$ (c \bmod n) := c - n \cdot \left \lfloor \frac{c}{n} \right \rfloor $$
whose value is the smallest non-negative integer in the same residue class modulo $n$. We observe that this term makes sense also for negative numbers $c$, provided that  $c \bmod n$ is positive by definition. 

Let $G_i(\vec \lambda, n, x)$ be the expression defined as follows. 
\begin{enumerate}
    \item In the polynomial $F_i(\vec \lambda, \vec x)$, every coefficient $c \in \mathbb Z$ is replaced by the term $(c \bmod n)$. 
    \item In the resulting polynomial, every parameter $\lambda_{e}$ is replaced by the term $(\lambda_{e} \bmod n)$. 
\end{enumerate}
Consider the Diophantine equation $G(\vec \lambda, n, \vec x, \vec y, \vec z) = 0$ defined as:
$$\sum _{i = 1}^k {\left ( {G_i(\vec \lambda,n, \vec x)}^2 - n^2 y_i^2 \right )}^2 + \sum _{j = 1}^m {(x_j + z_j - n + 1)}^2 = 0.$$
For given $\vec \lambda \in \mathbb Z^{s}$ and $n \geq 1$, define the set:
$$U_n(\vec \lambda) = \{(\vec x, \vec y, \vec z) \in \mathbb N^{2m + k} \,|\, G(\vec \lambda, n, \vec x, \vec y, \vec z) = 0 \}. $$ 
\begin{lemma}
    Given the system $S$, for all $\vec \lambda \in \mathbb Z ^s$ and $n \geq 1$, 
    $$|S_n(\vec \lambda) | = |U_n(\vec \lambda)| .$$
\end{lemma}

\begin{proof} This is again Lemma \ref{first} for the system of equations reduced modulo $n$. \end{proof}

\begin{lemma}\label{lem:sufficientt}
   There is an arithmetic term $t(n)$ such that given the system $S$, for all $\vec \lambda \in \mathbb Z^s$ and for all $n \geq 1$, 
   $$U_n(\vec \lambda ) \subset [0, t(n)]^{2m + k}.$$
\end{lemma} 

\begin{proof} In the equation $G(\vec \lambda, n, \vec x, \vec y, \vec z) = 0$, it is clear that the $x$-components and the $z$-components of every solution satisfy $x_j < n$ and $z_j < n$. We look now at the components $y_i$. We know that
$$y_i = \frac{|G_i(\vec \lambda, n, \vec x)|}{n}.$$
We perform the following actions on $G_i(\vec \lambda, n, \vec x)$:
\begin{enumerate}
    \item All subtraction signs are replaced with addition signs.
    \item All coefficients $(c \bmod n)$ are replaced by $n$, as we know that $0 \leq c \bmod n < n$.
    \item All occurrences of a parameter in the form $(\lambda_{e} \bmod n) $ are replaced by $n$, as we know that $0 \leq \lambda_e \bmod n < n$.
    \item All variables $x_j$ are replaced by $n$, as we know that as components of a natural solution, one has $x_j < n$. 
\end{enumerate} 
The result is a polynomial $H_i(n)$ with the property that $n \,|\,H_i(n)$ and
$$y_i \le \frac{H_i(n)}{n},$$
for every component $y_i$ of a solution. Now one chooses $t(n)$ equal to the sum of all these bounds:
$$t(n) = \frac{H_1(n)}{n}+ \dots + \frac{H_k(n)}{n} + n.$$
Of course $t(n)$ can be taken $1 + \max(H_1(n), \dots, H_k(n), n)$ and the $\max$ operation can be also expressed by an arithmetic term, but this expression is more complicated. 
\qed 

\begin{lemma}
    Given the system $S$, there is a constant $K \in \mathbb N$ such that for all parameters $\vec \lambda_1, \dots, \vec \lambda_k \in \mathbb Z$ and for all $n \geq 1$, 
    $$0 \leq G(\vec \lambda, n, \vec x, \vec y, \vec z) < 2^{n+K},$$
    for all $(\vec x, \vec y, \vec z) \in [0, t(n)]^{2m + k}$. 
\end{lemma} 

{\bf Proof}: For the function $G(\vec \lambda, n, \vec x, \vec y, \vec z)$ we apply almost the same steps as in the Lemma \ref{lem:sufficientt}, as follows: 
\begin{enumerate}
    \item All subtraction signs are replaced with addition signs.
    \item All coefficients $(c \bmod n)$ are replaced by $n$.
    \item All occurrences of a parameter $(\lambda_{e} \bmod n) $ are replaced by $n$.
    \item All variables $x_j$, $y_i$ and $z_j$ are replaced by $t(n)$. 
\end{enumerate} 
The result is a polynomial $H(n)$ such that for all parameters $\vec \lambda \in \mathbb Z^s$, for all $n \geq 1$, and for all $(\vec x, \vec y, \vec z) \in [0, t(n)]^{2m+k}$, 
$$0 \leq G(\vec \lambda,n, \vec x, \vec y, \vec z) \leq H(n). $$
But as $H(n)$ is a polynomial, there is always a constant $K$ such that for all $n \in \mathbb N$, 
$$H(n) < 2^{n+K}.$$
Indeed, $\lim_{n \rightarrow \infty} 2^n / |H(n)| = + \infty$, so there is some $N_0$ such that the inequality is surely true for $n \geq N_0$. For sufficiently large $K$, $2^K \times 2^n > H(n)$ for all $n \in \mathbb N$. 
\end{proof}

By putting all pieces together, we get the following theorem:

\begin{theorem}\label{main} There is an arithmetic term $f(\vec \lambda, n)$ such that
    for the system of polynomial equations $S$, for every values of the parameters $\vec \lambda \in \mathbb Z^s$ and for every $n \geq 1$, 
    $$f(\vec \lambda, n) = |S_n(\vec \lambda)|. $$
\end{theorem} 

{\bf Proof}: We apply the hypercube method from Section \ref{section2} for the equation:
$$G(\vec \lambda, n, \vec x, \vec y, \vec z) = 0$$
on the cube $[0, t(n)]^{2m+k}$ and using the exponential bound $2^{w(n)}$ where $w(n) = n+K$. The exponential polynomial $G(\vec \lambda, n, \vec x, \vec y, \vec z)$ is non-negative, as a sum of squares. We observe that:
$$G = \sum _\alpha C_\alpha(n, \lambda_1 \bmod n, \dots, \lambda_s \bmod n){\vec x}^{\vec \alpha_x}{\vec y}^{\vec \alpha_y}{\vec z}^{\vec \alpha_z}.$$
The variables are eliminated by summation and are replaced by $G$-functions depending on $n$.  \qed 

For the special case $n=1$, one sees whether the $k$-vector $(0,0,\dots, 0)$ belongs to the algebraic set. This is solved by observing whether all the free terms of the polynomials $F_i(\vec \lambda, \vec x)$ are zero or not. 

\section{Application to elliptic curves} 

We apply the method developed in Section \ref{sets} for the family of elliptic curves in short Weierstra{\ss } form
$$x_2^2 = x_1^3 + Ax_1 + B$$
with parameters $A, B \in \mathbb Z$ in the rings of remainder classes $\mathbb Z / n \mathbb Z$ with $n \geq 1$. We will construct an arithmetic term $f(A, B, n)$ representing the number of points $(x_1,x_2) \in \mathbb Z / n \mathbb Z \times \mathbb Z / n \mathbb Z$ which satisfy this equation. We have seen that we must count the solutions $(x_1, x_2, y, z_1, z_2) \in \mathbb N^5$ of the equation:
$$\left ( \left (x_2^2 -  x_1^3 - (A \bmod n) x_1 - (B \bmod n) \right )^2 - n^2 y^2\right )^2 + (x_1 + z_1 - n + 1)^2 + $$ $$ + (x_2 + z_2 - n + 1)^2 = 0.$$
Call the left-hand of the equation above $E(A, B, n, x_1, x_2, y, z_1, z_2) $. 

We observe that for every solution $(x_1, x_2, y, z_1, z_2) \in \mathbb N^5$, one has $0 \leq x_1, x_2, z_1, z_2 < n$. For $y$ we write:
$$y = \frac{|x_2^2 - x_1^3 - (A \bmod n) x_1 - (B \bmod n)|}{n}\le $$ $$\le \frac{x_2^2 + x_1^3 + (A \bmod n) x_1 + (B \bmod n)}{n} <$$
$$< \frac{n^2 + n^3 + n^2 + n}{n} = (n+1)^2.$$
As always $n < (n+1)^2$ for all $n \in \mathbb N$, we define $t(n) = (n+1)^2$. 

In order to find bounds for $E(A, B, n, x_1, x_2, y, z_1, z_2) $ as function defined on $[0, t(n)]^5$, we use a similar argument. For all $0 \leq x_1, x_2, y, z_1, z_2 \leq t(n)$ one has:
$$0 \leq E(A, B, n, x_1, x_2, y, z_1, z_2) \leq$$ $$\leq 
\left ( (t(n)^2 + t(n)^3 + t(n)^2 + t(n))^2 + n^2 t(n)^2 \right )^2 + 2 (2 t(n) + n - 1)^2.$$
Call the expression from the right-hand side of the second inequality $P(n)$. By direct computation we observe that for all $0 \leq n \leq 1000$,
$$P(n) < 2^{n + 90}$$
and it is easy to show that this inequality is true for all $n \in \mathbb N$, as the function $2^{n+90} / P(n)$ is increasing for $n > 1000$. So we define $w(n) = n + 90$. 

If we expand the expression $E(A, B, n, x_1, x_2, y, z_1, z_2)$, we get a sum of $49$ monomials in the variables $(x_1, x_2, y, z_1, z_2)$.

We observe that all the components of any solution are strictly smaller than $t(n) = (n+1)^2$, so we can take $t(n) - 1$ as upper limit of the Mazzanti summation. Summing up, 
we find the following:

\begin{theorem}
    The number of points of the affine algebraic curve defined by the parametric family:
    $$ x_2^2 = x_1^3 + Ax_1 + B$$
    with $A, B \in \mathbb Z$  (classically called short Weierstra{\ss } form) in $\mathbb Z / n \mathbb Z \times \mathbb Z / n \mathbb Z$ for $n \geq 1$ is given by the evaluation in $(A, B, n)$ of an arithmetic term of the shape:
    $$f(A, B, n) = \frac{\HW(M(A, B, n))}{n + 90} - (n+1)^{10}.$$ 
    Here the term $M(A, B, n)$ is a sum between the contribution of the free term of the equation (given above) and other $48$ sub-terms, any of them being the product of a coefficient from the Diophantine equation and five $G$-functions. 
\end{theorem} 

Consider that we keep the variable ordering $x_1, x_2, y, z_1, z_2$ to construct the term.

According to Subsection \ref{subsection:numberofsolutions}, the contribution of the free term $2 n^2 - 4 n  + 2 + (B \bmod n)^4 $ reads:
$$ C\left (2 n^2 - 4 n  + 2 + (B \bmod n)^4, 5, (n+1)^2, n+90 \right ) = $$ $$ = \frac{ \left ( 2^{n + 90} - 2n^2 + 4n - (B \bmod n)^4 -1 \right ) \left ( 2^{2(n+90)  {(n+1)}^{10} } - 1\right )} {2^{n+90 } + 1 }.$$ 
Also, for the monomial $-2 n^2 x_1^6 y^2$, the corresponding contribution reads: 
$$ A\left (-2 n^2 x_1^6 y^2, 5, (n+1)^2, n+90 \right ) = $$ $$= + 2 n^2 \left( 2^{n+90} - 1 \right)  G_{6} \left( 2^{ 2 (n + 90) } , n^2 + 2n \right) G_{0} \left ( 2^{ 2 (n+90) {(n+1)}^2 } , n^2 + 2n \right ) $$
$$ \times G_{2} \left ( 2^{ 2 (n+90) (n+1)^4 } , n^2 + 2n \right ) G_{0} \left ( 2^{ 2 (n+90) (n+1)^6 } , n^2 + 2n\right ) $$ $$ \times G_{0} \left ( 2^{ 2 (n+90) (n+1)^8 } , n^2 + 2n \right ).$$
By inspecting the exponents arising in the Diophantine equation, we observe that the following functions must be used to write down the term $M(A, B, n)$ and implicitly the term $f(A, B, n)$: $G_k$ with $k = 0, 1, 2, \dots, 9, 10$ and $G_{12}$.

\section{Conclusions}
\label{sec:conclusions}

We have shown that for every family of algebraic sets represented by a system of polynomial equalities with parameters $\vec \lambda \in \mathbb Z$, there is a closed formula $f(\vec \lambda, n)$ computing the exact number of points over the ring $\mathbb Z / n \mathbb Z$. These formulas are arithmetic terms with a fixed number of operations. Although the formulas are too complicated to be of practical use, their existence is interesting. We do not know how simple equivalent formulas can be. We treat also the example of elliptic curves in Weierstra{\ss } normal form.

\paragraph*{Acknowledgments.} This paper is dedicated to Professor Adrian Atanasiu on his 80-th birthday.


\begin{thebibliography}{12} 

\bibitem{istrateprunescushuniauniversal} G. Istrate, M. Prunescu and J. M. Shunia, \emph{Undecidability, Chaos and Universality in Arithmetic Terms.} arXiv e-prints (2026),
https://arxiv.org/abs/2606.09336

\bibitem{mazzanti2002plainbases} S. Mazzanti,
\emph{Plain bases for classes of primitive recursive functions}, Math. Logic Quart. \textbf{48} (2002), no.~1, 93--104.

\bibitem{marchenkov2007superposition} S.~S. Marchenkov,
\emph{Superpositions of elementary arithmetic functions},
J. Appl. Ind. Math. \textbf{1} (2007), 351--360.

\bibitem{matiyasevich1993hilbert} Y. Matiyasevich, \emph{Hilbert’s Tenth Problem.} MIT press (1993)

\bibitem{prunescu2024numbertheoreticfunctions} M. Prunescu and L. Sauras-Altuzarra, \emph{On the Representation of Number-Theoretic Functions by Arithmetic Terms} (2024),
https://arxiv.org/abs/2407.12928

\bibitem{prunescushuniaprimes} M. Prunescu and J. M. Shunia, \emph{On arithmetic terms expressing the prime-counting function and the n-th prime.} arXiv e-prints (2024),
https://arxiv.org/abs/2412.14594 

\bibitem{prunescushuniagcd} M. Prunescu and J. M. Shunia, \emph{Arithmetic-term representations for the
greatest common divisor.} INTEGERS, Electronic Journal of combinatorial number theory 26, A\#96, (2026)

\bibitem{robinson1952arithmetic} J. Robinson,
\emph{Existential definability in arithmetic},
Transactions of the American Mathematical Society \textbf{72} (1952), 437--449.

\end{thebibliography}
\end{document}